\documentclass[12pt]{amsart}
\usepackage[latin2]{inputenc}
\usepackage[english]{babel}
\usepackage{amsthm}
\usepackage{verbatim}

\usepackage{amsmath}
\usepackage{amssymb}
\usepackage{hyperref}

\newcommand{\ok}{$\omega$-categorical}

\newcommand{\aut}{\operatorname{Aut}}
\newcommand{\ba}{\mathrm{BA}}

\newcommand{\R}{\mathcal{R}}

\newcommand{\ol}{\overline}

\theoremstyle{plain}
\newtheorem{theorem}{Theorem}
\newtheorem{lemma}[theorem]{Lemma}
\newtheorem{proposition}[theorem]{Proposition}
\newtheorem{corollary}[theorem]{Corollary}

\theoremstyle{definition}
\newtheorem{definition}[theorem]{Definition}

\newcommand{\ftwo}{\mathbb{F}_2^{\lambda}}
\newcommand{\ff}{\mathbb{F}}
\newcommand{\gl}{\operatorname{GL}}
\newcommand{\sym}{\operatorname{Sym}}
\newcommand{\id}{\operatorname{id}}

\newcommand{\codim}{\operatorname{codim}}

\newcommand{\RM}{\operatorname{RM}}
\newcommand{\symdif}{\mathbin{\triangle}}

\newcommand{\almostsum}{\overline{\Sigma}}

\theoremstyle{remark}
\newtheorem{remark}[theorem]{Remark}

\begin{document}

\title{ Infinitely many reducts of homogeneous structures}

\author[B. Bodor]{Bertalan Bodor }
\address{E\"{o}tv\"{o}s Lor\'{a}nd University, 
          Department of Algebra and Number Theory,
         1117 Budapest,  P\'{a}zm\'{a}ny P\'{e}ter s\'{e}t\'{a}ny 1/c, Hungary}
\email{bodorb@cs.elte.hu}

\author[P. J. Cameron]{Peter J. Cameron}
\address{School of Mathematics and Statistics, University of St Andrews,
North Haugh, St Andrews, Fife KY16 9SS, UK}
\email{pjc20@st-andrews.ac.uk}

\author[Cs. Szab\'{o}]{Csaba Szab\'{o}}
\thanks{The research was supported by the Hungarian OTKA K109185 grant.}
\address{E\"{o}tv\"{o}s Lor\'{a}nd University,
         Department of Algebra and Number The\-o\-ry,
         1117 Budapest,  P\'{a}zm\'{a}ny P\'{e}ter s\'{e}t\'{a}ny 1/c, Hungary}
\email{csaba@cs.elte.hu}

\date{}

\begin{abstract}  It is shown that the countably infinite dimensional pointed vector space (the vector space equipped with a constant) over a finite field has infinitely many first order definable reducts. This implies that the countable homogeneous Boolean-algebra has infinitely many reducts. Our construction over the
$2$-element field is related to the Reed--Muller codes.
 
\end{abstract}

\maketitle

\section{Introduction}

We consider structures over finite relational languages (for short,
\emph{relational structures}). A relational structure $\Gamma$
is a \emph{reduct} of the structure $\Delta$ if they have the same domain
and every relation of $\Gamma$ has a first-order definition in
$\Delta$. Two structures are called first-order equivalent if each is
a reduct of the other. The relation ``is a reduct of'' is clearly transitive,
and so induces a partial order on the class of structures on a given domain
over a given language.

A relational structure $\Delta$ is said to be \emph{homogeneous} if any
isomorphism between finite induced substructures can be extended to an
automorphism of $\Delta$. Note that a homogeneous structure has the property
that the number of isomorphism types of $n$-element substructure it contains
is bounded above by the exponential of a polynomial in $n$.

A structure $\Delta$ is \emph{$\omega$-categorical} if it is (up to isomorphism)
the unique countable model of its first-order theory, or equivalently, if its
$n$-types coincide with the orbits of its automorphism group on $n$-tuples.
Thus, for example, a homogeneous structure is $\omega$-categorical. If
$\Delta$ is $\omega$-categorical, then there is a bijection between reducts of
$\Delta$ and closed overgroups of $\aut(\Delta)$ in the symmetric group on the
domain (with the topology of pointwise convergence), see \cite{Hodges}. Thus,
finding reducts of $\Delta$ is equivalent to finding closed overgroups of
$\aut(\Delta)$. (Note that the ordering of
reducts is the reverse of the ordering of closed overgroups.)

Simon Thomas~\cite{thomas1} made the intriguing conjecture that a countable
homogeneous structure over a finite relational language has only finitely many
inequivalent reducts. This conjecture has been verified for
several well-known homogeneous structures, but there seems to be no
general progress towards proving it.  In \cite{cam} the reducts of the
dense linear order $(\mathbb Q,\leq)$ are determined (though the language of
reducts is not used there). Thomas himself determined the reducts of the
random graph (the unique countable homogeneous universal graph) and of random
hypergraphs~\cite{thomas1,thomas2}. Curiously, both the random graph
and the dense linear order have $5$ reducts.  In \cite{JUZI} it is shown that
the ``pointed'' linear order has 116 reducts. Thus adding a constant to
$(\mathbb Q,\leq)$, the number of reducts can increase significantly. The
Henson graphs (the countable homogeneous universal $K_n$-free graphs for
$n\ge3$~\cite{henson}) have no nontrivial reducts \cite{thomas1}. Later, in
\cite{BPT} and \cite{pibo}, a general technique was introduced to
investigate first order definable reducts of homogeneous structures on
a finite language. Although the strategy works only under very special
conditions, it was possible to determine all reducts in
some cases. Applying these techniques several structures have been
analyzed: the random poset \cite{pppp}, \cite{ppppsz}, 
and the random graph revisited \cite{BP}. They all have
finitely many reducts.  For the pointed Henson graphs $(H_n,C)$, Pongr\'acz
\cite{pong1} showed that for $n=3$ there are 13, while for $n>3$ there
are 16, reducts of $(H_n,C)$. The proof of this highly non-trivial result
requires all known tricks and techniques. Similarly, in \cite{r42}, the 42 
proper  reducts of the ordered random graph are determined in 42 pages.
 In \cite{r42} it is mentioned that we do
not even know how to show that the lattice of reducts has only
finitely many atoms, or no infinite ascending or descending chains.

In order to learn more, it seems to be unavoidable to test the conjecture for
more of the classical structures from model theory, independently of
whether or not we believe the conjecture. We note that the result of
Ahlbrandt and Ziegler \cite[Theorem 3.5]{AZ} addresses the same issue.

The countable dimensional vector spaces over finite fields and the countable
atomless Boolean algebra are \ok. They are not  homogeneous on a
finite relational language -- for vector spaces this can be seen by noting that,
if the maximum arity of relations in the language is $n$, then $n+1$ linearly
independent vectors and $n+1$ vectors with sum zero and all proper subsets
independent are isomorphic as substructures but not of the same type.
However, they are of finite signature, and the vector spaces share with
homogeneous structures the property that the number of $n$-types is bounded by
the exponential of a polynomial in $n$. The first-order
definable reducts of the countable vector space and the symplectic space over 
$\ff_2$ were determined in \cite{kebesza}. There are finitely many of them.
The proper reducts of the vector space  are the affine space and the stabilizer
of $0$ in the symmetric group. In addition the symplectic space, the vector
space endowed with a symplectic form $\cdot$ has one additional reduct, the
vector space with the ternary relation $$\{(a,b,c)\mid a\cdot b+a\cdot c+b\cdot c=0, a,b,c\text{ are linearly independent}\}.$$

In this paper we show that the statement of Thomas' conjecture is not true for 
$\omega$-categorical structures of finite signature. We present infinitely many reducts of the pointed vector spaces over finite fields, and of the homogeneous 
Boolean algebra.

Our construction for pointed vector spaces over finite fields can be 
re-formulated in terms of infinite analogues of the Reed--Muller codes
\cite{Muller,Reed}.

We are grateful to the referee for helpful comments, including the reference
to \cite{AZ}.

\section{The construction}\label{construction}

Let $V $ be a countably infinite dimensional vector space over the 2-element field, $\ff_2$ and $0\neq C\in V$. We shall investigate the pointed vector space $(V,C)$ that is obtained by adding $C$ as a constant to $V$. The automorphism group of $V$ is $\aut(V)= \gl(V)$ and the automorphism group of $(V,C)$ is the stabilizer of $C$ in $\gl(V)$ so $\aut(V,C) =\gl(V)_C$.

We are going to consider mappings which interchange the two elements of some cosets of $\langle C \rangle$. In order to specify this, we fix a 1-codimensional subspace $W<V$ not containing $C$. Then, $V=\langle W, C\rangle$ and $V=\{w, w + C\mid w\in W\}$, so any coset of $\langle C\rangle$ contains a unique vector in $W$. We can specify our maps by giving the set of vectors of $W$ in which the relevant cosets meet $W$.

For $\varphi\in \gl(V)_C$, let $W_{\varphi}=W^{\varphi}\cap W$. We have the following two cases:
\begin{itemize}
\item
If  $W_\varphi=W^\varphi=W$, then $(w+C)^\varphi=w^\varphi+C$ for every $w\in W$. In this case, let $\ol\varphi=\varphi$.
\item   
Otherwise,  $V=W_\varphi \cup (W\setminus W_\varphi) \cup (W^\varphi\setminus W_\varphi) \cup (V\setminus(W\cup W^\varphi)  $. In this case, let $ \ol \varphi $ be defined in the following way: 
\begin{enumerate}
\item[--] $v^{\ol\varphi}=v^\varphi$ if $v^\varphi \in  W_\varphi  \cup (W_\varphi+C)  $ and  
\item[--]$v^{\ol\varphi}=v^\varphi+C$ otherwise.
\end{enumerate}
\end{itemize}
The map $\ol\varphi$ is an automorphism of $(V,C)$ mapping $W$ to $W$. The map $\ol \varphi $ is uniquely determined by $\tilde \varphi= \ol \varphi|_W $. For $\sigma \in \aut(W)$ let $\sigma^V$ denote this extension:

\begin{enumerate}
\item[--] $v^{\sigma^V}=v^\sigma$ if $v\in  W  $ and  
\item[--]$v^{\sigma^V}=(v+C)^\sigma+C$ if $v\notin  W  $   
\end{enumerate}

Now, let $U<W$ such that $U$ has codimension $2$ in $V$ (that is, $|V:U|=4$), and let $\sigma\in \aut W$. Then there is a unique $\sigma_U \neq \sigma^V \in \aut (V,C)$ defined by
\begin{enumerate}
\item[--]$v^{\sigma_U}=v^{\sigma^V}+C$ if $v^{\sigma^V} \in (W\setminus U) \cup(W\setminus U+C)  $.
\item[--] $v^{\sigma_U}=v^{\sigma^V}$ otherwise.
\end{enumerate}

\begin{definition}\label{flip} 
Let $S\subseteq W$. Let $h_S\in \sym(V)$ defined by 
\[
v^{h_S}=
\begin{cases}
v+C & \text{if $v\in S\cup (S+C)$},\\
v & \text{otherwise}.
\end{cases}
\]
If $v\in S\cup (S+C)$ then we say that \emph{$h_S$ flips $v$}.
\end{definition}

Set $H=\{h_S\mid S\subseteq W\}$. Now $H$ is a subgroup of $\sym(V)$ isomorphic to $Z_2^W$. (This follows immediately from Lemma~\ref{calc}(dif) below.)
With the above notations, $\sigma_U=\sigma^V\circ h_{W\setminus U} $. As noted, the subgroup $H$ is an elementary Abelian 2-group (in particular, it is commutative), and it is normalized by $\aut(W)^V:=\{\sigma^V\mid \sigma\in \aut(W)\}$. So there is a canonical embedding
$ \aut (W) \rtimes Z_2^W \to  \langle \aut(V,C), H \rangle  $ given by $(\sigma,h)\to \sigma^V\circ h $. Moreover, with the above notations we have
$\varphi=\ol\varphi\circ h_{W\setminus W_\varphi} $ for every
$\varphi\in \aut (V,C)$.

\begin{definition}
Given a subspace $W_n$ of codimension $n$ in $W$, let $H_n=\langle \aut(V,C), h_{W_n}\rangle$. 
\end{definition}

Since $\aut(V,C)$ acts transitively on $n$-codimensional subspaces of $W$, we see that $H_n$ does not depend on $W_n$, only on the codimension $n$.

\begin{lemma}[Subspace calculus]\label{calc} 
%Let $V $ be a vector space over $\ff_2$ and let $\dim V=\lambda> n$ where $\lambda $ is either  finite or $\lambda=\omega$
Let $n$ be a natural number.
\begin{enumerate}

\item[(dif)]\label{dif} Let $X,Y\subseteq W$. Then $h_X\circ h_Y = h_{X\symdif Y}$, where $\symdif$ denotes the symmetric difference of subsets.
\item[(gen)]\label{gen} For every $h\in H_n\cap H$ there are $W_i< V$ ($i=1,2,\dots,k$) such that
$h=\prod h_{W_i}$.
\item[(aff)]\label{affin} If $W'<W $, $\dim(W/W')=n$ and $w\in W$, then $h_{W'+w}\in H_n$. 
 
\item[(trans)] If $h_X\in H_n$ and $w\in W$, then $h_{X+w}\in H_n$.

\item[(codim)]\label{codim} If  $w\in W'<W $ and $\dim(W/W')<n$,  then $h_{W'}, h_{W'+w}\in H_n$.
\item[(two)] If $a, b\in U<W $ and $\dim U=n+1$, then there is an element $g\in H_n$ such that $g|_U= h_{\{a,b\}}|_U$.

\item[(even)] If $S\subseteq  U<W $, where $|S|$ is divisible by 2 and $\dim U=n+1$, then there is an element $g\in H_n$ such that $g|_U= h_{S}|_U$.
\item[(one)]  If $a\in U<W $ and $\dim U=n$, then there is an element $g\in H_n$ such that $g|_U= h_{\{a\}}|_U$.
\item[(odd)]  If $S\subseteq U<W $ and $\dim U=n$, then there is an element $g\in H_n$ such that $g|_U= h_{S}|_U$.

\end{enumerate}

\end{lemma}

\begin{proof} If we compose the two group elements, the elements of the intersection of $X$ and $Y$ are flipped twice, hence are fixed by the composition.  The elements of $X\setminus Y$  and $Y\setminus X$ are flipped once. This gives (dif).

Item (gen) is obvious from the definition of $H_n$ and the fact that the group $H$ is normalized by $\aut(V,C)$.

For {(aff)}  if $w\in W'$ then   $h_{W'+w}=h_{W'}$ and we are done. If $w\notin W'$, then  let  $W'=\langle W_1, a\rangle$, where $a\notin W_1$. Then $\langle W_1, w\rangle$ and $\langle W_1, a+w\rangle$ are $n$-codimensional subspaces of $W$. Thus $h_{\langle W_1, w\rangle}$ and $h_{\langle W_1, a+w\rangle}\in H_n$. Now, $h_{\langle U_1, w\rangle}\circ h_{\langle W_1,a+w\rangle}=h_{W'+w}$.

For $m>n$ every $n$-codimensional subspace is the disjoint union of $m$-codimensional affine subspaces. Hence ({codim)} follows from  (aff) and (dif).

By (gen) let $h_X=\prod h_{W_i}$. Then by item (aff) $h_{W_i+w}\in H_n$ and $h_{X+x}=\prod h_{W_i+x}$ thus (trans) follows.

For (two) let $W'$ be an $n$-codimensional subspace of $W$ such that $U\cap W'=\{0,a+b\}$. Then $h_{W'+a}$ satisfies the conditions.

Item (even) easily follows from (two).

For (one) let $W'$ be an $n$-codimensional subspace of $V$ such that $U\cap W'=\{0\}$. Then $h_{W'+a}$ satisfies the conditions.

For (odd) let $S\subseteq U$ and $g=\prod\limits_{a\in S} h_{U+a}$, where $U$ is the subspace from (one). Then $g|_U=h_S|_U$.

\end{proof}

\begin{definition}\label{rndef} 
%Let $V $ be a countably infinite dimensional vector space over the 2-element field, $C\in V$ and let $W<V$ be a 1-codimensional subspace, as above, not containing $C$. 
Let $\R_n$ denote the relation consisting of all $2^n$-tuples
$(x_1,x_2,\dots,x_{2^{n}})$ such that $\{ x_i, x_i+C\,|\, 1\leq i\leq 2^{n}\}$ 
is an affine subspace of $V$ and $|\{ x_i |\, 1\leq i\leq 2^{n}  \} \cap W|$
is even.
\end{definition}

\begin{proposition}\label{rnfix} 
Let $n\geq 1$.  Then $H_n$ preserves $\R_m$ if and only if $m\geq n+1$.
\end{proposition}

\begin{proof} 
Assume that $m\leq n$. Let $W_{\ol n}$ be an $n$-codimensional and $U_m$ be an $m$-dimensional subspace of $W$ such that $W_{\ol n}\cap U_m=\{0\}$. Such a pair of subspaces exists by the conditions on the dimensions. Let $x_1,\dots, x_{2^m}$ be an enumeration of the elements of $U_m$. Clearly, $(x_1,\dots, x_{2^m})\in \R_m$. Let $g=h_{W_n}$. Now, $g\in H_n$ and $v^g=v+C$ holds only for $0$ from $U_m$ and for $0\neq  v\in U_m$ we have $v^h=v$. That is $(x_1^g,\dots, x_{2^m}^g)$ contains $2^m-1$, in particular  odd many elements from $W$ and so $g$ does not preserve $\R_m$.

For the other direction suppose $m\geq n+1$. Let $\varphi\in \aut (V,C) $ and  consider the  canonical  form $\varphi=\ol\varphi\circ h_{W\setminus W_\varphi} $.  The map $\ol\varphi$ preserves $W$, hence preserves $\R_k$, as well, for arbitrary $k$. By item (codim) of Theorem~\ref{calc}  we have $h_{W\setminus W_\varphi}\in H_n$, hence it is enough to show that $H_{W_1}$ preserves $\R_m$ for every $n$-codimensional subspace $W_1$ of $W$.

So suppose $(x_1,\dots, x_{2^m})\in \R_m$, and let  $S=\{x_1,\dots, x_{2^m}\}$. Then $S\cup (S+C)$ is an affine subspace of $V$ and  either $x_i\in W$ or $x_i+C\in W$. Let $S_C =\{x_i\,|\, x_i\notin W  \}$. Now, $U=(S_C+C)\cup (S\setminus S_C)$  is an $n$-dimensional affine subspace of $W$ and $S=U^{h_{S_C}}$, that is $S$ is obtained from the affine subspace $U$ by applying $h_{S_C}$.  Also, by definition of $\R_m$ we have that $|S\cap U|=|S\cap W|$ is even. Now, let $W_1$ be an arbitrary $n$-codimensional subspace of $W$ and $h=h_{W_1}$. Then $\dim(W_1\cap U)>1$, hence $|W_1\cap U|$ is even. Moreover,   $S^h=U^{h_{S_C}h} $ so by item (dif) of Theorem~\ref{calc} we have that $|S^h\cap W|=|S^h\cap U|=|U^{h_{S_C}h} \cap U|=|U^{h_{S_C}}\symdif (W_1\cap U)|=|S\symdif (W_1\cap U)|$ is even. We have $\{x_i,x_i+C\}=\{x_i^h,x_i^h+C\}$, hence 
$(x_1^{h},\dots, x^{h}_{2^m})\in \R_m$. Thus $h_{W_1}$ preserves $\R_m$ and by item (gen) of Lemma~\ref{calc} $H_{n}$ preserves $\R_m$, as well.

\end{proof}

\begin{proposition}\label{chain}
For the closure of the subgroups $H_i $ we have $\ol H_1\lneq \ol H_2\lneq \cdots\lneq \ol H_n\lneq \dots$ 
\end{proposition}

\begin{proof} By item (dim) of Lemma~\ref{calc} if $U<V$ and $m=\dim U\geq n$ then $h_U\in H_n$. By item (gen) of Lemma~\ref{calc} these elements generate 
$H_m$. Thus $H_m\leq H_n $ and so $\ol H_m\leq \ol H_n $ holds.
 By Theorem~\ref{rnfix} the group $H_n$ preserves the relations  $\R_{n+1}$  and so does its closure, $\ol H_n$. On the other hand $H_{n+1}$ does not preserve $\R_{n+1}$, hence $\ol H_{n+1}$ does not preserve  $\R_{n+1}$, either. Thus $H_n\neq H_{n+1}$ and the statement holds.

\end{proof}

\begin{remark}
	By using the observation that each automorphism $\varphi$ of $(V,C)$ can be written as $\varphi=\ol\varphi\circ h_{W\setminus W_\varphi}$, it is easy to see that $H_1=H_0=\aut(V,C)\cup \aut(V,C)h_W$. In particular $\aut(V,C)$ is a normal subgroup of index 2 in $H_1$.
\end{remark}

\begin{theorem}\label{infred}  The lattice of first-order definable reducts of the pointed homogeneous vector space $(\ff_2^\omega,C )$ contains an infinite descending chain. 
In particular $(\ff_2^\omega,C )$ has infinitely many  first-order definable reducts.
\end{theorem}

\begin{proof} The first-order definable reducts of a homogeneous structure are
in a one-to-one order-reversing correspondence with the closed supergroups of
its automorphism groups. Proposition~\ref{chain} implies the statement.
\end{proof}

Finally, we give the relational description of the groups $H_n$.

\begin{definition}
	We define the ternary relation $\almostsum$ as $$\almostsum(x,y,z):=\bigl\{(x,y,z)\in V^3\mid x+y+z\in \{0,C\}\bigr\}.$$
\end{definition}

\begin{proposition}\label{semidir}
	Let $\pi\in \sym(V)$. Then $\pi\in \langle \aut(V,C), H\rangle=\aut(V,C)H$ if and only if $\pi$ preserves $\almostsum$.
\end{proposition}

\begin{proof}
	The equality $\langle \aut(V,C), H\rangle=\aut(V,C)H$ follows from the fact that the group $H$ is normalized by $\aut(V,C)$.
	
	It is clear that every element of $H$ preserves the relation $\almostsum$. This implies the forward implication.
	
	Now, let us assume that $\pi$ preserves $\almostsum$. Using the notation at the beginning of this section it is easy to check that $\tilde{\pi}=\ol \pi|_W$ is well-defined and it is an automorphism of $W$. Then $\tilde{\pi}^V$ is an automorphism of $(V,C)$ and $(\tilde{\pi}^V)^{-1}\circ \pi=h_S$ for some $S\subseteq W$. In particular $(\tilde{\pi}^V)^{-1}\circ \pi\in H$, and thus $\pi\in \aut(V,C)H=\langle \aut(V,C), H\rangle$.
\end{proof}

\begin{theorem}\label{relational_description} 
Let $\pi\in \sym (V,C)$. Then $\pi\in  \ol{H}_n$ if and only if $\pi $ preserves the relations $\R_{n+1}$ and $\almostsum$.
\end{theorem}

\begin{proof} 
	The group $H_n$ preserves $\R_{n+1}$ by Proposition~\ref{rnfix}, and it preserves $\almostsum$ by Proposition~\ref{semidir}. 
	
	Now, let $\pi$ be a permutation preserving $\R_{n+1}$ and $\almostsum$. We have to show that $\pi\in \ol{H_n}$. By Proposition~\ref{semidir} it follows that $\pi$ can be written as $\pi=\varphi\circ h_S$ for some $\varphi\in \aut(V,C)$ and $S\subseteq W$. As automorphisms of $(V,C)$ preserve the relations $\R_{m}$, we may assume that $\pi=h_S$ for some $S\subseteq W$. 
	
	Let $U\leq V$ be a finite dimensional subspace of $V$. It is enough to show that there is an $h\in H_n$ such that $h|_U=\pi|_U$. Clearly we can assume that $C\in U$. Let $U_1=U\cap W$. If $\dim U_1\leq n$ then item (odd) of Lemma~\ref{calc} proves the statement, $H_n|_U= Z_2^{U_1}$. If $\dim U_1=n+1$ then the statement holds by item (even) of Lemma~\ref{calc}. Now, let $\dim U_1=m$, where $m\geq n+2$. We proceed by induction on $n+m$. 

	First, let us assume that $n=1$. We claim that in this case both $S$ and $W\setminus S$ are affine subspaces of $W$. (This implies that $\codim S=1$, hence $\pi=h_S\in H_1$.) Suppose first that $u,v,w\in S$. Since $\pi$ preserve $\R_2$ it follows that $|S\cap \{u,v,w,u+v+w\}|$ is even. Hence $u+v+w\in S$. Similarly if $u,v,w\in W\setminus S$, then $u+v+w\in W\setminus S$.

	Now, assume that $n\geq 2$, $m\geq n+2$, and the statement holds for $n+m-1$ and $n+m-2$. Let $\dim U_1=m$. Let $U_2\leq U_1\cap W$ such that $\dim U_2=m-1$, and let $w\in U_1\setminus U_2$. By the induction hypothesis there is an $h_{T}\in H_n$ such that ${h_T}|_{U_2}=\pi|_{U_2}$. The permutation $\pi\circ h_{T}^{-1}=h_{S\symdif T}$ fixes $U_2+C$ elementwise, and it preserves $\R_{n+1}$. Let $k=h_{(S\symdif T)+w}$. Then $k$ preserves $\R_{n+1}$ by the definition of $\R_{n+1}$, and it fixes $U_1\setminus U_2$ elementwise. By using item (trans) of Lemma~\ref{calc} it is enough to show that there is an element $h\in H_n$ such that $h|_{U_1}=k|_{U_1}$. Let $W_2$ be a $1$-codimensional subspace of $W$ containing $U_2$, but not containing $U_1$. We would like to apply the induction hypothesis for the vector space $V_2=\langle W_2,C\rangle$, the subspace $\langle U_2,C\rangle$, the relation $\R_n$ and the permutation $k$. For this it is enough to show that $k$ preserves $\R_n$ restricted to $\langle U_2,C\rangle$. So suppose $x_1,x_2\dots, x_{2^n}\in \langle U_2,C\rangle$ and $(x_1,\dots, x_{2^n})\in \R_n$ and let $X=\{x_1,\dots, x_{2^n}\}$. We have to show that $|X\cap X^k|$ is even. Let $Y=X\cup (X+w)$. Then $|Y\cap Y^k|$ is even since $k$ preserves $\R_{n+1}$. We know that $k|_{U_1\setminus U_2}=\id_{U_1\setminus U_2}$, hence $Y\cap Y^k=(X\cap X^k)\cup ((X+w)\cap (X+w)^k))=(X\cap X^k)\cup (X+w)$, therefore $|X\cap X^k|=|Y\cap Y^k|-|X+w|=|Y\cap Y^k|-2^n$, which is even. So we can apply the induction hypothesis for the vector space $V_2$, the subspace $\langle U_2,C\rangle$, the relation $\R_n$ and the permutation $k$. It implies that there are $n-1$-codimensional subspaces $Y_1,Y_2,\dots ,Y_t$ of $W_2$ such that for $h=\prod h_{Y_i}$ we have that $h|_{U_2}=k|_{U_2}$. The subspaces $Y_i$ are $n$-codimensional subspaces of $V$, hence $h\in H_n$. Moreover $h$ fixes all elements of $U_1\setminus U_2\subset W\setminus W_2$, hence $h|_{U_1}=k|_{U_1}$, and this is what we wanted to show.
\end{proof}

\section{Reed--Muller codes}

Our construction can be re-formulated in terms of infinite analogues of
Reed--Muller codes \cite{Muller,Reed}.

Our description of the Reed--Muller codes follows van Lint \cite{vanLint}.

A \emph{binary linear code} of length $N$ is a vector subspace of
$\mathbb{F}_2^N$. Vectors in this space can be regarded as functions from
an $N$-set to $\mathbb{F}_2$. For this application we take $N=2^n$, and
identify the set of coordinates with $V=\mathbb{F}_2^n$.

The \emph{Reed--Muller code} $\RM(r,n)$ can be described in two different ways:
\begin{enumerate}
\item[--] it consists of all the functions from $V$ to $\mathbb{F}_2$ which can
be represented by polynomials of degree at most $r$ in the coordinates;
\item[--] it is spanned by the characteristic functions of subspaces of
codimension~$r$ in $V$.
\end{enumerate}
We summarise a few properties of these codes.
\begin{enumerate}
\item[--] $\RM(r,n)$ has dimension $\sum_{i=0}^r\binom{n}{i}$ and minimum 
weight $2^{n-r}$;
\item[--] $\RM(r,n)^\perp=\RM(n-r-1,n)$, where orthogonality is with respect to
the standard inner product.
\end{enumerate}

\smallskip

Now we return to our reducts. The automorphism group
of the pointed vector space $(\ftwo,C)$ is a semidirect product of the space
of linear functions $V\to\mathbb{F}_2$ by the general linear
group $\gl(V)$, where $V=\ftwo/\langle C\rangle$. (For this group acts on the
quotient space as the automorphism group of $V$; the kernel of this action 
fixes every coset, and so can be represented by maps from $V$ to $\ff_2$, the
image of a coset being $0$ or $1$ according as the elements in this coset are
fixed or interchanged by the element concerned. To see that the extension
splits, choose a complement $W$ for $\langle C\rangle$ in $\ftwo$; elements of
$\aut(\ftwo,C)$ form the required complement.)
Thus, any closed
$\gl(V)$-invariant subspace $W$ of the space of functions
$\lambda\to\ff_2$ that contains all linear functions will define a closed subgroup $W\rtimes\gl(V)$ containing all automorphisms, and
hence a reduct.

Let $W_k$ be the closure of the vector space of functions
$f:V\to\mathbb{F}_2$ given by polynomials of degree at most $k$
in the coordinates (these are ``infinite RM codes''). Note that, since $x^2$
and $x$ are equal as functions, we have $W_1\le W_2\le\cdots$; these subspaces
are closed and $\gl(V)$-invariant. The inclusions are strict since, for
example, the polynomial of degree $k$ which is the product of $k$ distinct
indeterminates cannot be written as a polynomial of smaller degree. (See also
the following paragraph). So we have
a descending chain of reducts. Note that non-zero vectors in these subspaces
all have infinite support.

While there is no inner product defined on the
vector space of all functions from $V$ to $\mathbb{F}_2$,
we can define the ``standard inner product'' $u\cdot w$ whenever $u$ is a
vector with finite support. Now a function belongs to $W_k$ if and only if it
is orthogonal to the characteristic function of every $(k+1)$-dimensional
subspace of $V$. This holds for polynomials of degree $k$ by the same
argument as in the finite case. Then, as a convergent sequence of
polynomials (in the topology of pointwise convergence) is ultimately constant
on any $(k+1)$-dimensional subspace,  its limit is also orthogonal to every such
subspace. We also see the strict inclusion of the subspaces $W_k$ from this
argument: for the product of $k$ distinct indeterminates meets some
$k$-dimensional subspace in a single point, and so fails to be orthogonal to
all such subspaces, and cannot lie in $W_{k-1}$.

So $W_k\rtimes\gl(V)$ is a closed subgroup of
the symmetric group containing all automorphisms, and hence a reduct of the pointed vector space.

This argument also verifies the relational definition of the reducts given
earlier.

We remark that the two definitions of the finite-dimensional RM codes are no
longer equivalent in the infinite case: the space spanned by the characteristic
functions of $k$-dimensional subspaces contains elements of finite support and
is not closed.

\section{Corollaries}

There are two obvious ways to generalize the result of Theorem~\ref{infred}. One is to find a similar construction for vector spaces over finite fields of odd characteristic; the other is to find structures that have $(V,C)$ as a first-order definable reduct. We start with the second.

Let  $\ba=(B,\wedge,\vee,0,1,\neg)$ denote the countable atomless Boolean algebra. It is known that this structure is {\ok}. For $a,b\in \ba$ let $a + b$  denote the symmetric difference of $a$ and $b$.

\begin{theorem}\label{infredb}  The lattice of first-order definable reducts of the homogeneous Boolean algebra  contains an infinite descending chain. 
In particular it has infinitely many  first-order definable reducts.
\end{theorem}

\begin{proof} The vector space $(\ff_2^\omega,C )$ is isomorphic to $(B,+,0,1)$, hence it is   a reduct of the homogeneous Boolean algebra. Being a first-order definable reduct is transitive, so Theorem~\ref{infred} implies the statement.
 \end{proof}

For the case of finite fields of odd characteristic, the construction is analogous.  Let $V $ be a countably infinite dimensional vector space over the $p$-element field, $\ff_p$ and $0\neq C\in V$. Let us fix a 1-codimensional subspace $W<V$ not containing $C$. Then, $V=\langle W,C \rangle$ and $V=\{w, w+\lambda C \,|\, w\in W, \lambda\in\ff_p\}$. The automorphism group of $V$ is $\aut(V)= \gl(V)$ and the automorphism group of $(V,C)$, the pointed vector space is the stabilizer of $C$ in $\gl(V)$ so $\aut(V,C) =\gl(V)_C$.  Let $S\subseteq W$. Let $h_S\in \sym(V)$ defined by $v^{h_S}=v+C$ for $v\in S$ and  $v^{h_S}=v$ else. For  a subspace  $W_n\leq V$, where $W_n$ is an  $n$-codimensional subspace and let $H_n=\langle \aut(V,C), h_{W_n}   \rangle$. Again, $H_n$ does not depend on the choice of $W_n$. Let $\R_n$ denote the relation $(x_1,x_2,\dots,x_{p^{n}})$, where $\{ x_i, x_i+\lambda C\,|\, 1\leq i\leq p^{n}, \lambda\in \ff_p  \}$ is a subspace of $V$ and $\sum x_i \in W$.
As in the case of characteristic 2,  for the closure of the subgroups $H_i $ we have $\ol H_1\lneq \ol H_2\lneq \cdots\lneq \ol H_n\lneq \dots$. Also, for any $\pi\in \sym (V,C)$ we have that $\pi\in  \ol{H}_n$ if and only if $\pi $ preserves $\R_{n+1}$. We arrive at the conclusion:

\begin{theorem}\label{infredp}  The lattice of first-order definable reducts of the pointed homogeneous vector space $(\ff_p^\omega,C )$ contains an infinite descending chain.  In particular $(\ff_p^\omega,C )$ has infinitely many  first-order definable reducts.
\end{theorem}

\begin{corollary}\label{infredpp}  The lattice of first-order definable reducts of the pointed homogeneous vector space over a finite field $\ff_q$ contains an infinite descending chain.  In particular $(\ff_q^\omega,C )$ has infinitely many  first-order definable reducts.
\end{corollary}

\begin{proof} Let $p$ be the characteristic of the field $\ff_q$. The structure 
$(\ff_q^\omega,C , +)$, where we consider only the addition as an operation, is 
a reduct of the pointed vector space, and is
isomorphic to  the vector space $(\ff_p^\omega,C )$. The statement follows from Theorem~\ref{infred} and  Theorem~\ref{infredp}. 
\end{proof}

These results can also be shown using an analogue of the Reed--Muller
construction; we do not pursue this further.


\begin{thebibliography}{99}

\bibitem{AZ}
Gisela Ahlbrandt and Martin Ziegler, 
\emph{Invariant subgroups of ${}^VV$},
Journal of Algebra, \textbf{151} (1992), 26--38.

\bibitem{BCP}
M. Bodirsky, H. Chen and M. Pinsker,
\emph{The reducts of equality up to primitive positive interdefinability},
Journal of Symbolic Logic, \textbf{75}(4) (2010), 1249-1292.








\bibitem{BP}
M. Bodirsky and M. Pinsker,
\textit{Minimal functions on the random graph},
Israel Journal of Mathematics, \textbf{200}(1) (2014), 251--296. 

\bibitem{BPT}
M. Bodirsky, M. Pinsker and T. Tsankov,
\textit{Decidability of definability},
Journal of Symbolic Logic, \textbf{78}(4) (2013), 1036--1054.




\bibitem{pibo}
M. Bodirsky and M. Pinsker,
\emph{Reducts of Ramsey structures}, Model Theoretic Methods in Finite Combinatorics, \textbf{558}, Contemporary Mathematics,
American Mathematical Society
(2011), 489-519.






\bibitem{r42} M. Bodirsky, M. Pinsker and A. Pongr\'acz,
\emph{The 42 reducts of the random ordered graph}, 
Proceedings of the London Mathematical Society \textbf{111} (3) (2015), 591-632.

\bibitem{functional} B. Bodor, K. Kende and Cs. Szab\'o, 
\emph{Functional reducts of Boolean algebras},
arXiv preprint arXiv:1506.01314 (2015).



\bibitem{kebesza} B. Bodor, K. Kende and Cs. Szab\'o,  \emph{ Permutation groups containing infinite symplectic linear groups and reducts of linear spaces over the two element field}, Communications in Algebra, to appear.




\bibitem{cam}
P. J. Cameron,
\emph{Transitivity of permutation groups on unordered sets},
Mathematische Zeitschrift, 
\textbf{148} (1976) 127-139.

\bibitem{henson}
C. W. Henson,
\textit{A family of countable homogeneous graphs},
Pacific J.~Math. \textbf{38} (1971), 69--83.


\bibitem{Hodges}
W. A. Hodges,
{\it Model theory}, Cambridge University Press, Cambridge, 1993.

\bibitem{JUZI}
M. Junker and M. Ziegler,
\emph{The 116 reducts of $(\mathbb{Q}; <; a)$}, Journal of Symbolic Logic,
\textbf{74(3)} (2008) 861-884.

\bibitem{vanLint}
J. H. van Lint, \textit{Introduction to Coding Theory}, Graduate Texts in

Mathematics \textbf{86}, Springer, Berlin, 1992.


\bibitem{Mac11}
D. Macpherson,
\emph{A survey of homogeneous structures}, Discrete Mathematics,
\textbf{311(15)} (2011) 1599-1634.


\bibitem{Muller}
D. E. Muller, Application of Boolean algebra to switching circuit design
and to error detection,
\textit{IRE Trans. Electronic Computers} \textbf{3} (1954), 6--12.


\bibitem{ppppsz}
P. P. Pach, M. Pinsker, A. Pongr\'acz and Cs. Szab\'o,
\textit{A new transformation of partially ordered sets}, 
J. Comb. Theory A \textbf{120:7} (2013) 1450-1462




\bibitem{pppp}
P. P. Pach, M. Pinsker, G. Pluh\'ar, A. Pongr\'acz and Cs. Szab\'o, 
\emph{Reducts of the random partial order},
Advances in Mathematics, \textbf {67}  (2014),  94--120.




\bibitem{pong1}
A. Pongr\'acz,
\emph{Reducts of the Henson graphs with a constant},
Annals of Pure and Applied Logic (2013), to appear


\bibitem{Reed}
I. S. Reed,
\textit{A class of multiple-error-correcting codes and the decoding scheme},
Trans. IRE Professional Group on Information Theory, 
\textbf{4} (1954),38--49.


\bibitem{thomas1}
S. Thomas,
\emph{Reducts of the random graph},
Journal of Symbolic Logic,
\textbf{56}(1) (1991), 176-181.


\bibitem{thomas2}
S. Thomas,
\textit{Reducts of random hypergraphs},
Annals of Pure and Applied Logic,
\textbf{80}(2) (1996), 165-193.




\end{thebibliography}
\end{document}